\documentclass[12pt]{article}

\usepackage{ascmac}
\usepackage{amsthm}
\usepackage{amsmath}
\usepackage{amssymb}
\usepackage{amsfonts}
\usepackage{bm}

\newtheorem{thm}{Theorem}[section]
\newtheorem{lem}[thm]{Lemma}
\newtheorem{df}[thm]{Definition}
\newtheorem{cor}[thm]{Corollary}

\makeatletter
\@addtoreset{equation}{section}

\makeatother

\title{The multiple Dirichlet product and the multiple Dirichlet series}
\author{TOMOKAZU ONOZUKA}

\begin{document}
\baselineskip=17pt
\date{}
\maketitle
\renewcommand{\thefootnote}{}
\footnote{2010 \emph{Mathematics Subject Classification}: Primary 11M32; Secondary 11A25.}
\footnote{\emph{Key words and phrases}: Multiple zeta function, Multiple zeta star function, Multiple Dirichlet series, Multiple Dirichlet product, Multiple Dirichlet convolution, Zero-free region.}
\footnote{The author was supported by JSPS KAKENHI Grant Number 13J00312.}
\renewcommand{\thefootnote}{\arabic{footnote}}
\setcounter{footnote}{0}
\begin{abstract}
First, we define the multiple Dirichlet product and study the properties of it. From those properties, we obtain a zero-free region of a multiple Dirichlet series and a multiple Dirichlet series expression of the reciprocal of a multiple Dirichlet series.
\end{abstract}
\section{Introduction}

The Euler-Zagier multiple zeta function $\zeta_{EZ,k}(s_1, \ldots,s_k)$, the multiple zeta star function $\zeta_k^*(s_1, \ldots,s_k)$, and the Mordell-Tornheim multiple zeta function $\zeta_{MT,k}(s_1,\ldots,s_k;s_{k+1})$ are defined by
\begin{align}
  &\zeta_{EZ,k}(s_1,\ldots,s_k):=\sum_{0<m_1<m_2< \cdots<m_k}\frac{1}{m_1^{s_1}m_2^{s_2}\cdots m_k^{s_k}},
\label{eq:EZ}
\\
  &\zeta_k^* (s_1,\ldots,s_k) := \sum_{0<m_1\leq m_2\leq \cdots\leq m_k}\frac{1}{m_1^{s_1}m_2^{s_2}\cdots m_k^{s_k}},
\label{eq:str}
\\
  &\zeta_{MT,k}(s_1,\ldots,s_k;s_{k+1}):=\sum_{m_1,\ldots,m_k=1}^{\infty}
  \frac{1}{m_1^{s_1}\cdots m_k^{s_k}(m_1+\cdots+m_k)^{s_{k+1}}},
\label{eq:MT}
\end{align}
respectively, where $s_i\ (i=1, \ldots,k+1)$ are complex variables. Matsumoto \cite{mat1} proved that the series (\ref{eq:EZ}) and (\ref{eq:str}) are absolutely convergent in
\begin{align}
  \left\{(s_1,\ldots,s_k)\in\mathbb{C}^k\ |\ \Re(s_k(k-l+1))>l\ \ (l=1,\ldots,k)\right\}\label{eq:abs}
\end{align}
where $s_k(n)=s_n+s_{n+1}+ \cdots+s_k\ \ (n=1, \ldots,k)$. The series (\ref{eq:MT}) is absolutely convergent in
\[
  \left\{(s_1,\ldots,s_k;s_{k+1})\in \mathbb{C}^{k+1}\ |\ \Re (s_l)>1\ (l=1, \ldots,k),\ \Re(s_{k+1})>0\right\}.
\]
 Akiyama, Egami and Tanigawa \cite{aki1} and Zhao \cite{zhao}, independently of each other, proved the meromorphic continuation of the series (\ref{eq:EZ}) to the whole space. Akiyama, Egami and Tanigawa used the Euler-Maclaurin summation formula, while Zhao used generalized functions to prove the analytic continuation. Matsumoto \cite{mat3} proved the meromorphic continuation of the series (\ref{eq:MT}) to $\mathbb{C}^{k+1}$. The series (\ref{eq:str}) can be expressed by the sum of Euler-Zagier multiple zeta functions and the Riemann zeta function. (Note that the Riemann zeta function is one of Euler-Zagier multiple zeta functions.) For example, $\zeta^*_{2}$ and $\zeta^*_3$ can be expressed as a sum of Euler-Zagier multiple zeta functions as follows:
\begin{align*}
  \zeta^*_2(s_1,s_2)&=\zeta_{EZ,2}(s_1,s_2)+\zeta(s_1+s_2),\\
  \zeta^*_3(s_1,s_2,s_3)&=\zeta_{EZ,3}(s_1,s_2,s_3)+\zeta_{EZ,2}(s_1+s_2,s_3)+\zeta_{EZ,2}(s_1,s_2+s_3)\\
  &\qquad\qquad\qquad\qquad\qquad\qquad\qquad\qquad\qquad+\zeta(s_1+s_2+s_3).
\end{align*}
We can obtain the above expression by separating the sum of the series (\ref{eq:str}). From the above expression and the meromorphic continuation of Euler-Zagier multiple zeta functions, it follows that the multiple zeta star functions can be continued meromorphically to the whole space. 

In this paper, as a generalization of the classical notion of Dirichlet series, we define multiple Dirichlet series by
\begin{align}
  F(s_1,\ldots,s_k;f) := \sum_{m_1,\ldots,m_k=1}^\infty\frac{f(m_1,\ldots,m_k)}{m_1^{s_1}\cdots m_k^{s_k}},
\label{eq:mds}
\end{align}
where $f:\mathbb{N}^k\to\mathbb{C}$ and $(s_1,\ldots,s_k)$ is in the region of absolute convergence for $F(s_1,\ldots,s_k;f)$. The series (\ref{eq:mds}) is a generalization of the series (\ref{eq:EZ}), (\ref{eq:str}) and (\ref{eq:MT}). De la Bret\'eche \cite{bre} treated function (\ref{eq:mds}) in the case $f(m_1,\ldots,m_k)>0$.

In this paper, we study the multiple Dirichlet series defined by (\ref{eq:mds}). 
In Section 2, we consider a ring of multiple arithmetic functions with the usual addition $+$ and the multiple Dirichlet product $*$. In the case of one variable, Cashwell and Everett \cite{cas} proved that the ring is a UFD. In a way similar to that in \cite{cas}, we find that the ring is also a UFD in the multivariable case. In Section 3, we treat the multiple Dirichlet series (\ref{eq:mds}) and prove the main theorem (Theorem \ref{zfr}). The main theorem has two statements. In the first statement, we state about a zero-free region of $F(s_1,\ldots,s_k;f)$. In the second statement, we state that the reciprocal of $F(s_1,\ldots,s_k;f)$ has a multiple Dirichlet series expression $F(s_1,\ldots,s_k;f^{-1})$. Finally, by using the main theorem, we get a multiple Dirichlet series expression of the reciprocal of $\zeta_k^*(s_1,\ldots,s_k)$.

The author would like to express his thanks to Professor Kohji Matsumoto, Professor L\'aszl\'o T\'oth, and Miss Ade Irma Suriajaya for their valuable advice and comments.

\section{The ring of multiple arithmetic functions}

We call $f: \mathbb{N}^k\longrightarrow \mathbb{C}$ a multiple (k-tuple) arithmetic function,
and we define the set of $k$-tuple arithmetic functions by
\begin{equation*}
\Omega=\Omega_k:=\{f\ |\ f: \mathbb{N}^k\longrightarrow \mathbb{C}\}.
\end{equation*}
We define the set $U$ by;
\begin{equation*}
 U=U_k:=\{f\in\Omega\ |\ f(1,\ldots,1)\neq0\}.
\end{equation*}
We use bold letters to express $k$-tuple of integers like $\bm{a}=(a_1, \ldots,a_k)$. In particular, $\bm{1}$ means $(1, \ldots,1)$. Moreover we define the product of two $k$-tuple of integers $\bm{a}\cdot\bm{b}$ by $(a_1b_1, \ldots,a_kb_k)$.
\begin{df}\label{mdp}\
For $f,g\in\Omega$ and $\bm{n}\in \mathbb{N}^k$, we define the multiple Dirichlet product $*$ by
\begin{align*}
  (f*g)(\bm{n})=\sum_{\substack{\bm{a}\cdot\bm{b}=\bm{n}\\\bm{a},\bm{b}\in \mathbb{N}^k}}f(\bm{a})g(\bm{b}).
\end{align*}
\end{df}
If $k=1$, the multiple Dirichlet product is the well-known Dirichlet product. Hence the above product is a generalization of the Dirichlet product. In the present section, our purpose is to study the properties of the multiple Dirichlet product. The properties were studied by some mathematicians. In \cite{to}, T\'oth described the details of such studies. In particular, if we define an addition as $(f+g)(\bm{n}):=f(\bm{n})+g(\bm{n})$, he mentioned that $(\Omega_k,+,*)$ is an integral domain with the identity function $I$ which is defined by
\begin{align*}
  I(\bm{n}):=\left\{
  \begin{array}{ll}
  1\ \ &(\bm{n}=\bm{1}),\\
  0&({\rm otherwise}),
  \end{array}
  \right.
\end{align*}
and its unit group is $U$. For $f\in U$, $f^{-1}$ denotes the inverse function of $f$ with respect to the multiple Dirichlet product $*$. Then $f^{-1}$ is given by the following lemma.

\begin{lem}\label{thm3}\
For each $f\in U$, $f^{-1}\in U$ can be constructed recursively as follows;
\begin{align*}
  f^{-1}(\bm{1})&=\frac{1}{f(\bm{1})},\\
  f^{-1}(\bm{n})&=-\frac{1}{f(\bm{1})}\sum_{\substack{ \bm{a}\cdot\bm{b}=\bm{n}\\\bm{b}\neq\bm{n}}}f(\bm{a})f^{-1}(\bm{b})\ &(\bm{n}\neq1).
\end{align*}
\end{lem}
\begin{proof}
In the case $n_1+ \cdots+n_k=k$, i.e. $\bm{n}=\bm{1}$, we have
\begin{align*}
  \left(f*f^{-1}\right)(\bm{1})=f(\bm{1})f^{-1}(\bm{1})=1.
\end{align*}
Next, let $d>k$, and assume that $f^{-1}(\bm{n})$ are determined for all $\bm{n}\in\mathbb{N}^k$ which satisfies $n_1+ \cdots+n_k<d$. Then for $\bm{n}\in \mathbb{N}^k$ which satisfies $n_1+ \cdots+n_k=d$, we have 
\begin{align*}
  \left(f*f^{-1}\right)(\bm{n})&=\sum_{\bm{a}\cdot\bm{b}=\bm{n}}f(\bm{a})f^{-1}(\bm{b})\\
  &=f(\bm{1})f^{-1}(\bm{n})+
  \sum_{\substack{ \bm{a}\cdot\bm{b}=\bm{n}\\\bm{b}\neq\bm{n}}}f(\bm{a})f^{-1}(\bm{b})=0.
\end{align*}
\end{proof}

In this section, let us consider $(\Omega_k,+,*)$ and prove that $(\Omega_k,+,*)$ is a UFD.

First, we define a norm $N:\Omega\longrightarrow\mathbb{Z}_{\geq0}$ by
\begin{align*}
  N(f)=\begin{cases}
    0  &(f=0),\\
    \underset{\bm{n}\in\mathbb{N}^k}{\min\ }\{n_1\cdots n_k\ |\ f(\bm{n})\neq0\}  &(f\neq0).
  \end{cases}
\end{align*}
By the definition of the norm, we can show that $N(f)=1$ if and only if $f\in U$ holds.
\begin{thm}\label{norm}
For $f,g\in\Omega$, we have
\[
  N(f*g)=N(f)N(g).
\]
\end{thm}
\begin{proof}
In the case $f=0$ or $g=0$, since $f*g=0$, we have $N(f*g)=N(f)N(g)=0$. Hence we consider the case $f,g\neq0$. If $n_1 \cdots n_k<N(f)N(g)$, then we have
\begin{align*}
  (f*g)(\bm{n})&=\sum_{\bm{a}\cdot\bm{b}=\bm{n}}f(\bm{a})g(\bm{b})=0,
\end{align*}
since $a_1 \cdots a_k<N(f)$ or $b_1 \cdots b_k<N(g)$. Therefore it is sufficient to find $\bm{n}$ such that $(f*g)(\bm{n})\neq0$ and $n_1\cdots n_k=N(f)N(g)$. We define $(M_1, \ldots,M_k), (L_1, \ldots,L_k) \in \mathbb{N}^k$ as follows;
\begin{align*}
  M_1&:= \min\{m_1\ |\ f(m_1,\ldots,m_k)\neq0,\ m_1 \cdots m_k=N(f)\},\\
  M_2&:= \min\{m_2\ |\ f(M_1,m_2,\ldots,m_k)\neq0,\ M_1m_2 \cdots m_k=N(f)\},\\
  &\ \ \vdots\\
  M_k&:= \min\{m_k\ |\ f(M_1,\ldots,M_{k-1},m_k)\neq0,\ M_1\cdots  M_{k-1}m_k=N(f)\},\\
  L_1&:= \min\{l_1\ |\ g(l_1,\ldots,l_k)\neq0,\ l_1 \cdots l_k=N(g)\},\\
  L_2&:= \min\{l_2\ |\ g(L_1,l_2,\ldots,l_k)\neq0,\ L_1l_2 \cdots l_k=N(g)\},\\
  &\ \ \vdots\\
  L_k&:= \min\{l_k\ |\ g(L_1,\ldots,L_{k-1},l_k)\neq0,\ L_1\cdots  L_{k-1}l_k=N(g)\}.
\end{align*}
Next we prove
\begin{align}
  (f*g)(\bm{M}\cdot\bm{L})&=\sum_{\bm{a}\cdot\bm{b}=\bm{M}\cdot\bm{L}}f(\bm{a})g(\bm{b})\notag\\
  & =f(\bm{M})g(\bm{L})\neq0.\label{MLmethod}
\end{align}
For $\bm{a}\neq\bm{M}$ with $a_1\cdots a_k=N(f)$, there exists an index $i$ such that $a_i\neq M_i$ and $a_j=M_j$ for all $j<i$. If $a_i<M_i$, by the definition of $M_i$, $f(\bm{a})=0$. If $a_i>M_i$, then $b_i<L_i$ and $b_j=L_j$ for $j<i$, so $g(\bm{B})=0$. Hence (\ref{MLmethod}) holds.

Hence we have $N(f*g)=N(f)N(g)$.
\end{proof}

Next, we define an equivalence relation among functions belonging to $\Omega$. For $f,g\in\Omega$, we write $f\sim g$ and say $f$ is equivalent to $g$ if there exists a function $\varepsilon\in U$ such that $f=\varepsilon *g$. Also we write $f|g$ if $g=f*h$ for some $h\in\Omega$. 
\begin{thm}
Suppose $f,g\in\Omega$. Then $f$ is equivalent to $g$, if and only if $f|g$ and $g|f$. 
\end{thm}
\begin{proof}

Assume $f\sim g$. Then there exists a function $\varepsilon\in U$ such that $f=\varepsilon *g$, that is, $g|f$. Since $\varepsilon\in U$, $\varepsilon^{-1}*f=g$ holds. This means $f|g$.

Next, we assume $f|g$ and $g|f$. If $f=0$, then $g$ is also $0$. Therefore $f=\varepsilon*g$ holds for any $\varepsilon\in U$. Hence we are only left to consider the case $f\neq0$.
We see that $f|g$ and $g|f$ if and only if there exist $\alpha,\beta\in\Omega$ such that $f= \alpha*g$ and $g= \beta*f$, and if so, then $N(f)=N(\alpha)N(g)$ and $N(g)=N(\beta)N(f)$. From the above equations, it follows that $N(\alpha)N(\beta)=1$. It implies $\alpha,\beta\in U$. Hence $f$ is equivalent to $g$.
\end{proof}

For $p\in\Omega_k\setminus U$, we call $p$ {\it prime} if $p=f*g$ implies $f\in U$ or $g\in U$. We define $P_k=P$ to be the set of all prime functions. If $N(f)$ is prime in $\mathbb{N}$, then $f$ is prime in $\Omega$. Hence we can find infinitely many primes. A multiple arithmetic function $f\in\Omega$ is said to be {\it composite} if $f$ satisfies $f\neq0$, $f\notin U$, and $f\notin P$. If $f\sim g$, then $f\in P$ implies $g\in P$. This property indicates that the equivalence relation $\sim$ preserves primitivity. The same property holds for $0$, $U$, and composite functions, respectively.

Next, we show that each non-zero function $f\in\Omega\setminus U$ can be decomposed into a finite product of prime functions. In the case $f\in P$, the function $f$ itself is the finite product of prime functions. Hence we consider the case $f\notin P$. By the assumption, there exist multiple arithmetic functions $g,h\in\Omega\setminus U$ such that $f=g*h$. Then the inequalities $1<N(g),N(h)<N(f)$ hold. If both of $g$ and $h$ are primes, then $g*h$ is the prime factorization of $f$. If $g$ is a composite function, then there exist multiple arithmetic functions $g_1,g_2\in\Omega\setminus U$ such that $g=g_1*g_2$, and the inequalities $1<N(g_1),N(g_2)<N(g)$ hold. Repeating the above algorithm, we can obtain the prime factorization of $f$. Because the norm is a non-negative integer, the product is finite.

Let $P_\mathbb{N}$ be the set of all prime numbers in $\mathbb{N}$. We define a set $\mathcal{P}_k=\mathcal{P}$ by
\begin{align*}
  \mathcal{P}_k=\mathcal{P}:=\{(1,\ldots,1,\underset{j}{p},1,\ldots,1)\in\mathbb{N}^k\ |\ p\in P_\mathbb{N},\ 1\leq j\leq k\}.
\end{align*}
We take a bijection $S:\mathbb{N}\longrightarrow\mathcal{P}$. Since $\mathcal{P}$ is a countably infinite set, we can take the bijection $S$. We put $\bm{p}_j:=S(j)$. Then each $\bm{m}\in \mathbb{N}^k$ can be decomposed into a finite product of $\mathcal{P}$. We define maps $\alpha_j\ (j\in\mathbb{N})$ as follows; 
\begin{align*}
\bm{m}=:\bm{p}_1^{\alpha_1(\bm{m})}\cdot\bm{p}_2^{\alpha_2(\bm{m})}\cdot\cdots.
\end{align*}
The definition of $\alpha_j$ implies the equation $\alpha_j(\bm{a}\cdot\bm{b})=\alpha_j(\bm{a})+\alpha_j(\bm{b})$. By using $\alpha_j$, we define a map $R:\Omega\longrightarrow\mathbb{C}_\omega:=\mathbb{C}\{x_1,x_2, \ldots\}$ by $R(f)=\sum_{m_1,\ldots,m_k=1}^{\infty}f(\bm{m})x_1^{\alpha_1(\bm{m})}x_2^{\alpha_2(\bm{m})}\cdots$ for $f\in\Omega$,
where $\mathbb{C}_\omega$ is the set of the infinite series of the above form, where each term has only a finite number of $x_j$. Note that although the series $R(f)$ contains $k$, the set $\mathbb{C}_\omega$ does not depend on $k$.
\begin{lem}\label{iso}
The map $R$ is an isomorphism.
\end{lem}
\begin{proof}
The equation $R(f+g)=R(f)+R(g)$ is trivial. In addition, we have
\begin{align*}
  R(f*g)&=\sum_{m_1,\ldots,m_k=1}^{\infty}\sum_{\bm{a}\cdot\bm{b}=\bm{m}}f(\bm{a})g(\bm{b})
  x_1^{\alpha_1(\bm{a})+\alpha_1(\bm{b})}x_2^{\alpha_2(\bm{a})+\alpha_2(\bm{b})}\cdots\\
  &=\left(\sum_{\bm{a}}f(\bm{a})x_1^{\alpha_1(\bm{a})}x_2^{\alpha_2(\bm{a})}\cdots\right) 
  \left(\sum_{\bm{b}}g(\bm{b})x_1^{\alpha_1(\bm{b})}x_2^{\alpha_2(\bm{b})}\cdots\right)\\
  &=R(f)R(g).
\end{align*}
Hence $R$ is homomorphism.

Next, we show that $R$ is a bijection. Trivially $f\neq g$ implies $R(f)\neq R(g)$. For any $A\in \mathbb{C}_\omega$, we can constitute a multiple arithmetic function $f_A\in\Omega$ by using coefficients of $A$ such that $R(f_A)=A$.
\end{proof}

In the case of one variable, Cashwell and Everett \cite{cas} proved that the ring $(\Omega_1,+,*)$ is a UFD by first showing that $\Omega_1$ is isomorphic to $\mathbb{C}_\omega$ and then showing that $(\mathbb{C}_{\omega},+,\times)$ is a UFD. Here, for the multivariable case, we use Lemma \ref{iso} and the result of Cashwell and Everett that $(\mathbb{C}_{\omega},+,\times)$ is a UFD, and we obtain the following theorem. 
\begin{thm}
$(\Omega_k,+,*)$ is a UFD.
\end{thm}

In addition to the above theorem, Lemma \ref{iso} implies an another theorem. Lemma \ref{iso} states that $\Omega_k$ is isomorphic to $\mathbb{C}_\omega$ for all $k\in\mathbb{N}$. This fact means that $\Omega_k$ is isomorphic to $\Omega_l$ for $k,l\in\mathbb{N}$.
\begin{thm}
All of $\Omega_k$ are isomorphic, i.e. $\Omega_k\cong \Omega_l$ holds for all $k,l\in\mathbb{N}$.
\end{thm}


\section{The multiple Dirichlet series}
In this section, we consider the zero-free region of the multiple Dirichlet series and the reciprocal of the multiple Dirichlet series by using the notion of multiple Dirichlet product. The following theorem is a basic property of the multiple Dirichlet series.
\begin{thm}\label{pro}(\cite{to}, Proposition 10)
For $f,g\in\Omega_k$, we have
\begin{align*}
  &F(s_1,\ldots,s_k;f)+F(s_1, \ldots,s_k;g)=F(s_1, \ldots,s_k;f+g),\\
  &F(s_1,\ldots,s_k;f)F(s_1, \ldots,s_k;g)=F(s_1, \ldots,s_k;f*g),
\end{align*}
where $(s_1, \ldots,s_k)$ lies in the region of absolute convergence for the series $F(s_1,\ldots,s_k;f)$ and $F(s_1,\ldots,s_k;g)$.
\end{thm}

\begin{cor}\label{R}\
Let $f\in U$. If $F(s_1,\ldots,s_k;f)$ and $F(s_1, \ldots,s_k;f^{-1})$ are absolutely convergent on $R\subset\mathbb{C}^k$, then $F(s_1, \ldots,s_k;f)$ has no zeros on $R$.
\end{cor}
\begin{proof}
Let $(s_1, \ldots,s_k)\in R$. Then by applying Theorem $\ref{pro}$, we have
\begin{align*}
  F(s_1, \ldots,s_k;f)F(s_1, \ldots,s_k;f^{-1})=F(s_1, \ldots,s_k;I)=1.
\end{align*}
\end{proof}

To find the zero-free region of the series $F(s_1, \ldots,s_k;f)$, we have to find the region of absolute convergence $R\subset\mathbb{C}^k$, and to find the region $R$, we have to evaluate $f^{-1}$. For this purpose, we prepare the following lemma.
\begin{lem}\label{3.4}
For $\alpha>1$, we have
\begin{align*}
  \sum_{d|n}d^{ \alpha}\leq\zeta( \alpha)n^{ \alpha}.
\end{align*}
\end{lem}
\begin{proof}
We have
\begin{align*}
  \sum_{d|n}d^{ \alpha}=n^\alpha\sum_{d|n}\left(\frac{d}{n}\right)^{ \alpha}=
  n^\alpha\sum_{d|n}\frac{1}{d^{ \alpha}}\leq\zeta( \alpha)n^{ \alpha}.
\end{align*}
\end{proof}
By the above lemma, we can evaluate the function $f^{-1}$.
\begin{thm}\label{eva}\
Let $f\in U$ satisfy the condition that there exist constants $C>0$ and $r_1,\ldots,r_k\in\mathbb{R}$ such that $|f(\bm{n})|\leq Cn_1^{r_1}n_2^{r_2}\cdots n_k^{r_k}$ for $\bm{n}\neq\bm{1}$. We put $\alpha_j>1+r_j$ $(j=1, \ldots,k)$ satisfying $\zeta(\alpha_1-r_1)\zeta( \alpha_2-r_2) \cdots\zeta( \alpha_k-r_k)\leq 1+|f(\bm{1})|/C$. Then we have
\begin{align*}
  |f^{-1}(\bm{n})|\leq \frac{n_1^{ \alpha_1}n_2^{ \alpha_2}\cdots n_k^{ \alpha_k}}{|f(\bm{1})|}.
\end{align*}
\end{thm}
\begin{proof}
We use induction on $n_1+ \cdots+n_k$. In the case $n_1+ \cdots+n_k=k$, i.e. $\bm{n}=\bm{1}$, then $f^{-1}(\bm{1})=1/f(\bm{1})$, so we have
\begin{align*}
  |f^{-1}(\bm{1})|=\frac{1}{|f(\bm{1})|}.
\end{align*}
Next, let $d>k$, and assume that $ |f^{-1}(\bm{n})|\leq n_1^{ \alpha_1}n_2^{ \alpha_2}\cdots n_k^{ \alpha_k}/|f(\bm{1})|$ for all $\bm{n}\in \mathbb{N}^k$ which satisfies $n_1+ \cdots+n_k<d$. Then for $\bm{n}\in\mathbb{N}^k$ which satisfies $n_1+ \cdots+n_k=d$, by using Lemma \ref{3.4}, we have 
\begin{align*}
  |f^{-1}(\bm{n})|
  &\leq\left|\frac{1}{f(\bm{1})}\right|
  \sum_{\substack{\bm{a}\cdot\bm{b}=\bm{n}\\\bm{a},\bm{b}\in\mathbb{N}^k\\\bm{b}\neq\bm{n}}}
  |f(\bm{a})||f^{-1}(\bm{b})|\\
  &\leq\frac{C}{|f(\bm{1})|^2}
  \sum_{\substack{\bm{a}\cdot\bm{b}=\bm{n}\\\bm{b}\neq\bm{n}}}
  a_1^{r_1}b_1^{ \alpha_1} \cdots a_k^{r_k}b_k^{ \alpha_k}\\
  &=\frac{C}{|f(\bm{1})|^2}\left\{n_1^{r_1}\cdots n_k^{r_k}
  \left(\sum_{b_1|n_1}b_1^{ \alpha_1-r_1}\right)\cdots
  \left(\sum_{b_k|n_k}b_k^{ \alpha_k-r_k}\right)
  -n_1^{ \alpha_1}\cdots n_k^{ \alpha_k}\right\}\\
  &\leq\frac{C}{|f(\bm{1})|^2}\left(\zeta( \alpha_1-r_1)n_1^{ \alpha_1}\cdots
  \zeta( \alpha_k-r_k)n_k^{ \alpha_k}-n_1^{ \alpha_1}n_2^{ \alpha_2}\cdots n_k^{ \alpha_k}\right)\\
  &\leq  \frac{n_1^{ \alpha_1}n_2^{ \alpha_2}\cdots n_k^{ \alpha_k}}{|f(\bm{1})|}.
\end{align*}
\end{proof}
Finally, we obtain the main theorem.
\begin{thm}\label{zfr}\
Let $f$ and $\alpha_1, \ldots, \alpha_k$ satisfy the conditions in Theorem $\ref{eva}$. Then $F(s_1, \ldots,s_k;f)$ and $F(s_1, \ldots,s_k;f^{-1})$ have no zeros in the region
\[
  \left\{(s_1,\ldots,s_k)\in \mathbb{C}^k\ |
  \ \Re (s_j)>1+\alpha_j\ \ (j=1, \cdots,k)\right\}.
\]
Moreover, in the same region $F(s_1, \ldots,s_k;f)$ and $F(s_1, \ldots,s_k;f^{-1})$ satisfy the relation
\begin{align*}
  (F(s_1, \ldots,s_k;f))^{-1}=F(s_1,\ldots,s_k;f^{-1}).
\end{align*}
\end{thm}
\begin{proof}
Since $f(\bm{n})\ll n_1^{r_1}n_2^{r_2}\cdots n_k^{r_k}$, $F(s_1, \ldots,s_k;f)$ is convergent absolutely in
\[
  \left\{(s_1,\ldots,s_k)\in \mathbb{C}^k\ |
  \ \Re (s_j)>1+r_j\ \ (j=1, \ldots,k)\right\}.
\]
Since $f^{-1}(\bm{n})\ll n_1^{ \alpha_1}\cdots n_k^{ \alpha_k}$ by Theorem \ref{eva}, $F(s_1, \ldots,s_k;f^{-1})$ is convergent absolutely in
\[
  \left\{(s_1,\ldots,s_k)\in \mathbb{C}^k\ |
  \ \Re (s_j)>1+\alpha_j\ \ (j=1, \ldots,k)\right\}.
\]
Hence we obtain the theorem by using Theorem $\ref{R}$.
\end{proof}


Next, we consider a restricted multiple Dirichlet series. To express the series (\ref{eq:EZ}), (\ref{eq:str}) and (\ref{eq:MT}), we define three functions as follows;
\begin{align*}
  &u_{EZ}(\bm{n}):=\begin{cases}
  1\ \ &(n_1<n_2< \cdots<n_k),\\
  0&({\rm otherwise}),
  \end{cases}\\
  &u^*(\bm{n}):=\begin{cases}
  1\ &(n_1\leq n_2\leq \cdots\leq n_k),\\
  0&({\rm otherwise}),
  \end{cases}\\
  &u_{MT}(\bm{n}):=\begin{cases}
  1\ &(n_{k+1}=n_1+n_2+\cdots+n_{k}),\\
  0&({\rm otherwise}).
  \end{cases}
\end{align*}
Then the series (\ref{eq:EZ}), (\ref{eq:str}) and (\ref{eq:MT}) can be expressed by
\begin{align*}
  &\zeta_{EZ,k}(s_1,\ldots,s_k)=F(s_1,\ldots,s_k;u_{EZ}),
\\
  &\zeta_k^* (s_1,\ldots,s_k)=F(s_1,\ldots,s_k;u^*),
\\
  &\zeta_{MT,k}(s_1,\ldots,s_k;s_{k+1})=F(s_1,\ldots,s_{k+1};u_{MT}).
\end{align*}

In \cite{aki2}, Akiyama and Ishikawa treated the multiple $L$-function which is defined by
\begin{align*}
  L_k(s_1,\ldots,s_k|\chi_1,\ldots,\chi_k):=\sum_{m_1<m_2< \cdots<m_k}
  \frac{\chi_1(m_1)\cdots\chi_k(m_k)}{m_1^{s_1}m_2^{s_2}\cdots m_k^{s_k}},
\end{align*}
where $\chi_j$ are Dirichlet characters. This is a generalization of the series (\ref{eq:EZ}). 
Moreover in \cite{mat4}, Matsumoto and Tanigawa treated the series
\begin{align*}
  \sum_{m_1,\ldots,m_k=1}^{\infty}\frac{a_1(m_1)a_2(m_2) \cdots a_k(m_k)}{m_1^{s_1}(m_1+m_2)^{s_2}\cdots (m_1+ \cdots+m_k)^{s_k}},
\end{align*}
where series $\sum_{m=1}^{\infty}a_j(m)/m^s\ (1\leq j \leq k)$ satisfy good conditions. This is also a generalization of the series (\ref{eq:EZ}). 

To consider such series, we generalize the series (\ref{eq:EZ}) as follows;
\begin{align*}
  \sum_{0<m_1<m_2< \cdots<m_k}\frac{f(m_1, \ldots,m_k)}{m_1^{s_1}m_2^{s_2}\cdots m_k^{s_k}}.
  \label{eq:segen}
\end{align*}
This series is a restriction of the multiple Dirichlet series (\ref{eq:mds}). To treat the restricted multiple Dirichlet series, we define three sets by
\begin{align*}
  \Omega_{EZ}&:=\{f\in\Omega\ |\ f(\bm{n})=0\ {\rm for}\ \bm{n}\ {\rm which\ does\ not\ satisfy}\ n_1<\cdots<n_k\},\\
  \Omega^*&:=\{f\in\Omega\ |\ f(\bm{n})=0\ {\rm for}\ \bm{n}\ {\rm which\ does\ not\ satisfy}\ n_1\leq\cdots\leq n_k\},\\
  \Omega_{MT}&:=\{f\in\Omega\ |\ f(\bm{n})=0\ {\rm for}\ \bm{n}\ {\rm which\ satisfies}\ n_k<n_1+n_2+\cdots+n_{k-1}\}.
\end{align*}
\begin{thm}\label{subring}
The ring $(\Omega^*,+,*)$ is a subring of the ring $(\Omega,+,*)$ which contains the unit group $\Omega^*\cap U$. The rings $(\Omega_{EZ},+,*)$ and $(\Omega_{MT},+,*)$ are subrings of the ring $(\Omega,+,*)$ which do not contain the unit group.
\end{thm}
\begin{proof}
Trivially all of the sets are closed with respect to the addition $+$. Hence it is sufficient to consider the Dirichlet product $*$.

First we consider the set $\Omega^*$. Let $f,g\in\Omega^*$. For $\bm{n}$ which does not satisfy $n_1\leq\cdots\leq n_k$, we have
\begin{align*}
  (f*g)(\bm{n})=\sum_{\bm{a}\cdot\bm{b}=\bm{n}}f(\bm{a})g(\bm{b}).
\end{align*}
Then, for each pair of $\bm{a}$ and $\bm{b}$ which satisfies $\bm{a}\cdot\bm{b}=\bm{n}$, at least one of them does not satisfy the inequalities $a_1\leq\cdots\leq a_k$ and $b_1\leq\cdots\leq b_k$. It implies that $(f*g)(\bm{n})=0$. Therefore $f*g\in \Omega^*$ holds. Next we assume $f\in \Omega^*\cap U$. Suppose that there exists an $\bm{n}$ which satisfies $f^{-1}(\bm{n})\neq0$ and does not satisfy $n_1\leq\cdots\leq n_k$. We choose such $\bm{n}$ which satisfies the condition that $n_1+\cdots+n_k$ is the smallest. Similarly to the above proof, an equality
\begin{align*}
  f^{-1}(\bm{n})=-\frac{1}{f(\bm{1})}\sum_{\substack{ \bm{a}\cdot\bm{b}=\bm{n}\\\bm{b}\neq\bm{n}}}f(\bm{a})f^{-1}(\bm{b})
\end{align*}
implies that $f^{-1}(\bm{n})=0$. It contradicts the assumption $f^{-1}(\bm{n})\neq0$. Hence $f^{-1}\in \Omega^*\cap U$ holds.

The statement that $(\Omega_{EZ},+,*)$ forms a subring can be proven similarly as above.

Next we prove the statement that $(\Omega_{MT},+,*)$ forms a subring. Let $f,g\in \Omega_{MT}$. For $\bm{n}$, we have
\begin{equation*}
  (f*g)(\bm{n})=\sum_{\bm{a}\cdot\bm{b}=\bm{n}}f(\bm{a})g(\bm{b})\label{eq:wa1}
\end{equation*}
If $\bm{a}$ and $\bm{b}$ satisfy the inequalities $a_k\geq a_1+\cdots+a_{k-1}$ and $b_k\geq b_1+\cdots+b_{k-1}$, the inequality 
\begin{align*}
  (a_1b_1+ \cdots+a_{k-1}b_{k-1})^2&\leq(a_1^2+\cdots+a_{k-1}^2)(b_1^2+\cdots+b_{k-1}^2)\\
  &\leq(a_1+\cdots+a_{k-1})^2(b_1+\cdots+b_{k-1})^2\\
  &\leq a_k^2b_k^2
\end{align*}
implies the inequality $n_k\geq n_1+\cdots+n_{k-1}$. Hence for $\bm{n}$ which satisfies $n_k<n_1+n_2+\cdots+n_{k-1}$, at least one of the inequalities $a_k<a_1+\cdots+a_{k-1}$ and $b_k<b_1+\cdots+b_{k-1}$ holds. Therefore the equation $(f*g)(\bm{n})=0$ is valid for $\bm{n}$ which satisfies $n_k<n_1+n_2+\cdots+n_{k-1}$. It implies that $f*g\in \Omega_{MT}$.
\end{proof}

By the above theorem, the product of two restricted multiple Dirichlet series is also a restricted multiple Dirichlet series, that is
\begin{align*}
  &\left(\sum_{0<m_1<\cdots<m_k}\frac{f(m_1,\ldots,m_k)}{m_1^{s_1}\cdots m_k^{s_k}}\right)
  \left(\sum_{0<n_1<\cdots<n_k}\frac{g(n_1,\ldots,n_k)}{n_1^{s_1}\cdots n_k^{s_k}}\right)\\
  &\qquad\qquad\qquad\qquad\qquad\qquad\qquad\qquad\qquad
  =\sum_{0<n_1<\cdots<n_k}\frac{(f*g)(n_1,\ldots,n_k)}{n_1^{s_1}\cdots n_k^{s_k}},\\
  &\left(\sum_{0<m_1\leq\cdots\leq m_k}\frac{f(m_1,\ldots,m_k)}{m_1^{s_1}\cdots m_k^{s_k}}\right)
  \left(\sum_{0<n_1\leq\cdots\leq n_k}\frac{g(n_1,\ldots,n_k)}{n_1^{s_1}\cdots n_k^{s_k}}\right)\\
  &\qquad\qquad\qquad\qquad\qquad\qquad\qquad\qquad\qquad
  =\sum_{0<n_1\leq\cdots\leq n_k}\frac{(f*g)(n_1,\ldots,n_k)}{n_1^{s_1}\cdots n_k^{s_k}},\\
  &\left(\sum_{m_k\geq m_1+\cdots+m_{k-1}}\frac{f(m_1,\ldots,m_k)}{m_1^{s_1}\cdots m_k^{s_k}}\right)
  \left(\sum_{n_k\geq n_1+\cdots+n_{k-1}}\frac{g(n_1,\ldots,n_k)}{n_1^{s_1}\cdots n_k^{s_k}}\right)\\
  &\qquad\qquad\qquad\qquad\qquad\qquad\qquad\qquad\qquad
  =\sum_{n_k\geq n_1+\cdots+n_{k-1}}\frac{(f*g)(n_1,\ldots,n_k)}{n_1^{s_1}\cdots n_k^{s_k}}.
\end{align*}
In addition, the inverse of a restricted multiple Dirichlet series is also a restricted multiple Dirichlet series, that is
\[
\left(\sum_{0<m_1\leq\cdots\leq m_k}\frac{f(m_1,\ldots,m_k)}{m_1^{s_1}\cdots m_k^{s_k}}\right)^{-1}
=\sum_{0<m_1\leq\cdots\leq m_k}\frac{f^{-1}(m_1,\ldots,m_k)}{m_1^{s_1}\cdots m_k^{s_k}}
\]
for $f\in\Omega^*\cap U$.

In \cite{mat3}, Matsumoto defined the Apostol-Vu multiple zeta function by
\begin{align*}
  \zeta_{AV,k-1}(s_1,\ldots,s_{k-1};s_k):=\sum_{0<m_1<\cdots<m_{k-1}}
  \frac{1}{m_1^{s_1}\cdots m_{k-1}^{s_{k-1}}(m_1+\cdots+m_{k-1})^{s_{k}}}.
\end{align*}
 To express this series, we define a function $u_{AV}$ as follows;
\begin{align*}
  u_{AV}(\bm{n}):=\begin{cases}
  1\ &(0<n_1<\cdots<n_k\ {\rm and}\ n_{k}=n_1+n_2+\cdots+n_{k-1}),\\
  0&({\rm otherwise}).
  \end{cases}
\end{align*}
Then $u_{AV}\in\Omega_{EZ}\cap\Omega_{MT}$ holds. If we define $\Omega_{AV}=\Omega_{EZ}\cap\Omega_{MT}$, then it follows from Theorem \ref{subring} that $\Omega_{AV}$ is a subring of the ring $(\Omega,+,*)$ which does not contain the unit group.

As mentioned in the introduction, the series (\ref{eq:str}) is absolutely convergent in the region (\ref{eq:abs}). This fact improves Theorem \ref{zfr}.
\begin{thm}\label{zfr2}
Let $f\in\Omega^*\cap U$ and $\alpha_1, \ldots, \alpha_k$ satisfy the conditions in Theorem $\ref{eva}$. Then $F(s_1, \ldots,s_k;f)$ and $F(s_1, \ldots,s_k;f^{-1})$ have no zeros in the region
\[
  \left\{(s_1,\ldots,s_k)\in \mathbb{C}^k\ |
  \ \Re (s_k(k-l+1))>l+\alpha_k(k-l+1)\ \ (l=1, \ldots,k)\right\}
\]
where $\alpha_k(l)=\alpha_l+\alpha_{l+1}+ \cdots+\alpha_k\ \ (l=1, \ldots,k)$. Moreover, in the same region $F(s_1, \ldots,s_k;f)$ and $F(s_1, \ldots,s_k;f^{-1})$ satisfy the relation
\begin{align*}
  (F(s_1, \ldots,s_k;f))^{-1}=F(s_1,\ldots,s_k;f^{-1}).
\end{align*}
\end{thm}
\begin{proof}
Since $f(\bm{n})\ll n_1^{r_1}n_2^{r_2}\cdots n_k^{r_k}$, $F(s_1, \ldots,s_k;f)$ is absolutely convergent in
\[
  \left\{(s_1,\ldots,s_k)\in \mathbb{C}^k\ |
  \ \Re (s_k(k-l+1))>l+r_k(k-l+1)\ \ (l=1, \ldots,k)\right\},
\]
where $r_k(l)=r_l+r_{l+1}+ \cdots+r_k\ \ (l=1,\ldots,k)$. Since $f^{-1}(\bm{n})\ll n_1^{ \alpha_1}\cdots n_k^{ \alpha_k}$ by Theorem \ref{eva}, $F(s_1, \ldots,s_k;f^{-1})$ is convergent absolutely in
\[
  \left\{(s_1,\ldots,s_k)\in \mathbb{C}^k\ |
  \ \Re (s_k(k-l+1))>l+\alpha_k(k-l+1)\ \ (l=1, \ldots,k)\right\}.
\]
Hence we obtain the theorem by using Theorem $\ref{R}$.
\end{proof}

Applying Theorem \ref{zfr}, Theorem \ref{zfr2} and the identity theorem, we have the following corollary.
\begin{cor}\label{last}
 $\zeta_{EZ,k}(s_1,\ldots,s_k)+1$ and $\zeta_k^*(s_1,\ldots,s_k)$ have no zeros in the region
\[
  S:=\left\{(s_1,\ldots,s_k)\in \mathbb{C}^k\ |
  \ \Re (s_k(k-l+1))>l+\alpha_k(k-l+1)\ \ (l=1, \ldots,k)\right\}
\]
where $\alpha_i>1$ $(i=1, \ldots,k)$ satisfy $\zeta(\alpha_1)\zeta( \alpha_2) \cdots\zeta( \alpha_k)\leq 2$. In the same region $S$, the reciprocal of the two functions have multiple Dirichlet series expressions
\begin{align*}
  (\zeta_{EZ,k}(s_1,\ldots,s_k)+1)^{-1}&=F(s_1,\ldots,s_k;(u_{EZ}+I)^{-1}),\\
  (\zeta_k^* (s_1,\ldots,s_k))^{-1}&=F(s_1,\ldots,s_k;(u^*)^{-1}),
\end{align*}
respectively. Furthermore $F(s_1,\ldots,s_k;(u_{EZ}+I)^{-1})$ and $F(s_1,\ldots,s_k;(u^*)^{-1})$ are continued meromorphically to the whole space.
\end{cor}

In the last corollary we obtained the zero free region of $\zeta^*_k$. This is not the best possible region. Similar to \cite[equation (2)]{Spi}, we find that the region
\[
  S':=\left\{(s_1,\ldots,s_k)\in \mathbb{C}^k\ |
  \ \zeta_k^* (\Re s_1,\ldots,\Re s_k)<2,\ \Re(s_k(k-l+1))>l\ \ (l=1,\ldots,k) \right\}
\]
is also the zero free region, and this region has a point which is not contained in $S$, that is, $S'\setminus S\neq \emptyset $. For example, when $k=2$, $(s_1,s_2)=(2,2)\notin S$, since for $(s_1,s_2)\in S$, $s_2$ needs to satisfy the condition $\Re(s_2)>2$. On the other hand, $(2,2)\in S'$, since by \cite[corollary 2.3]{Hof} we have
\begin{align*}
  \zeta_2^*(2,2)&=\zeta_{EZ,2}(2,2)+\zeta(4)=\frac{\pi^4}{120}+\frac{\pi^4}{90}=\frac{7\pi^4}{360}<2.
\end{align*}
Hence $(2,2)\in S'\setminus S$, so $S$ is not the best possible region. However, $S$ has a simpler expression than $S'$. This is an advantage of Corollary \ref{last}.


Graduate School of Mathematics\\
Nagoya University\\
Chikusa-ku, Nagoya 464-8602, Japan\\
E-mail:\ m11022v@math.nagoya-u.ac.jp


\begin{thebibliography}{9}

\bibitem{aki1} S. Akiyama, S. Egami and Y. Tanigawa, {\it Analytic continuation of multiple zeta-functions and their values at non-positive integers}, Acta Arith, \mbox{\boldmath $98$} (2001), 107--116.
\bibitem{aki2} S. Akiyama and H. Ishikawa, {\it On analytic continuation of multiple $L$-functions}, Analytic Number Theory, C. Jia and K. Matsumoto, eds. (2002), 1--16.
\bibitem{cas} E. D. Cashwell and C. J. Everett, {\it The ring of number-theoretic functions}, Pacific J. Math. \mbox{\boldmath $9$} (1959), 975--985.
\bibitem{bre} R. de la Bret\'eche, {\it Estimation de sommes multiples de fonctions arith\'etiques}, Compositio Mathematica \mbox{\boldmath $128$} (2001), 261--298.
\bibitem{Hof} M. Hoffman, {\it Multiple harmonic series}, Pacific J. Math.,  \mbox{\boldmath $152$} (1992), 275--290.
\bibitem{mat1} K. Matsumoto, {\it On analytic continuation of various multiple zeta-functions}, Number Theory for the Millenium (Urbana, 2000), Vol. II, M. A. Bennett et. al. (eds.), A. K. Peters, Natick, MA, 2002, pp. 417--440.
\bibitem{mat3} K. Matsumoto, {\it On Mordell-Tornheim and other multiple zeta-functions}, In: Proc. Session in Analytic Number Theory and Diophantine Equations, (eds. D. R. Heath-Brown and B. Z. Moroz), Bonner Math. Schriften, \mbox{\boldmath $360$}, Bonn, 2003, n.25, 17pp.
\bibitem{mat4} K. Matsumoto and Y. Tanigawa, {\it The analytic continuation and the order estimate of multiple Dirichlet series}, J. Th\'eorie des Nombres de Bordeaux, \mbox{\boldmath $15$} (2003), 267--274.
\bibitem{Spi} R. Spira, {\it Zero-free regions of} $\zeta^{(k)}(s)$, J. London Math. Soc., \mbox{\boldmath $40$} (1965), 677-682.
\bibitem{to} L. T\'oth, {\it Multiplicative arithmetic functions of several variables: a survey}, preprint, arXiv:1310.7053.
\bibitem{zhao} J. Zhao, {\it Analytic continuation of multiple zeta functions}, Proc. Amer. Math. Soc. \mbox{\boldmath $128$} (2000), 1275--1283.


\end{thebibliography}
\end{document}